\documentclass[11pt,reqno]{amsart}
\usepackage[body={7in,9in},centering]{geometry}
\usepackage{amsmath,amssymb,amsthm,mathrsfs}
\usepackage{color,xcolor}
\definecolor{cobalt}{RGB}{61,99,181}
\usepackage[colorlinks,citecolor=cobalt,linkcolor=cobalt]{hyperref}
\usepackage{float}
\usepackage{exscale}
\usepackage{relsize}
\usepackage{graphicx}
\usepackage{tikz}
\newtheorem{thm}{Theorem}[section]
\newtheorem{defi}[thm]{Definition}

\newtheorem{lem}[thm]{Lemma}
\newtheorem{rem}[thm]{Remark}

\numberwithin{equation}{section}

\date{\today}

\makeatletter

\newcommand{\Rmnum}[1]{\expandafter\@slowromancap\romannumeral #1@}
\makeatother

\begin{document}

\title[Toeplitz operators]{Toeplitz operators on $\mathcal L^p$-spaces of a tree}

\author[Mingmei Huang]{Mingmei Huang}
\address{
\textsuperscript{1}
 College of Mathematics and Statistics, Chongqing University, Chongqing, 401331, P. R. China}
\email{huangmm0909@163.com}

\author[Xiaoyan Zhang]{Xiaoyan Zhang}
\address{
\textsuperscript{2}
 College of Mathematics and Statistics, Chongqing University, Chongqing, 401331, P. R. China}
\email{xiaoyanzhang@cqu.edu.cn}

\author[Xianfeng Zhao]{Xianfeng Zhao}
\address{
\textsuperscript{3}
College of Mathematics and Statistics, Chongqing University, Chongqing, 401331, P. R. China}
\email{xianfengzhao@cqu.edu.cn}



\begin{abstract}
 Let $T$ be a rooted, countable infinite tree without terminal vertices. In the present paper, we characterize the spectra, self-adjointness and positivity of Toeplitz operators on the spaces of  $p$-summable functions on $T$. Moreover, we obtain a necessary and sufficient condition for Toeplitz operators to have finite rank on  such function spaces.
\end{abstract} \maketitle

\section{Introduction}\label{Intro}

Let us begin with the preliminary definitions, concepts and relevant notation on trees.  A \emph{graph} $G$ is a pair $(V,E)$ consisting of a countably-infinite set of \emph{vertices} $V$ and a set of \emph{edges} $E\subset \big\{(u, v): u, v\in V, u\neq v\big\}$ (which is a subset of   $V\times V$).  If $(u, v)\in E$, then  we say that $u$ and $v$ are \emph{adjacent} (or \emph{neighbors}) and we denote this by $u\thicksim v$. For each vertex $v$, the \emph{degree} of $v$ is the number of vertices adjacent to $v$, which is denoted  by $\mathrm{deg}(v)$. We say that $G$ is \emph{locally finite} if the degree of every vertex is finite. We call that a vertex $v$ is a \emph{terminal vertex} if $v$ has a unique neighbor. A \emph{path of length $n$} joining two vertices $u$ and $v$ is a finite sequence of $(n+1)$ distinct vertices $$u=u_0\thicksim u_1\thicksim u_2 \thicksim\cdots \thicksim u_n=v.$$

 By a  \emph{tree} $T$ we mean a locally finite, connected and simply-connected graph, i.e., for each pair of vertices there is one and only one path between them. We shall identify $T$ with the collection of its vertices $V$.  In the present paper, the tree we consider has a distinguished vertex, which we call the root of $T$ and we denote it by $o$. The length of a path joining the root $o$ and a vertex $v$ is called the \emph{length} of $v$ and is denoted by $|v|$.

Given a tree $T$ rooted at $o$ and a vertex $v\in T$, a vertex $w$ is called a \emph{descendant}  of $v$ if $v$ lies in the unique path from $o$ to $w$. In this case, the vertex $v$ is called an \emph{ancestor} of $w$. We denote by $v^{-}$ the neighbor of a vertex $v$ which is an ancestor of $v$, and  the vertex $v$ is called a \emph{child} of $v^{-}$. For each $w\in T$, we denote the set of all children of $w$ by $\mathrm{Chi}(w)$. Given a vertex $v$, the \emph{sector determined by $v$} is the set $S_v$ consisting of $v$ and all its descendants. In particular, $S_o=T$.

Let $1\leqslant p<\infty$. The function space $\mathcal L^p(T)$ of $T$ is defined as the set of  functions $f: T\rightarrow \mathbb C$ such that
$$\sum_{v\in T}|f(v)|^p<\infty.$$
It is easy to show that every  $\mathcal{L}^p(T)$ is a Banach space when endowed with the norm
$$\|f\|_p:=\bigg(\sum_{v\in T}|f(v)|^p\bigg)^{\frac{1}{p}},$$
since it is isomorphic to the  $p$-summable sequence  space $\ell^p(V)$.
 Moreover,  the dual space of $\mathcal{L}^p(T)$ with $1<p<\infty$ can be identified with $\mathcal{L}^q(T)$ under the sesquilinear dual  pairing
$$\langle f, g\rangle:=\sum_{v\in T}f(v)\overline{g(v)},\ \ \ f\in \mathcal{L}^p(T),\ \   g\in \mathcal{L}^q(T),$$
where $q=\frac{p}{p-1}$ is the conjugate exponent of $p$. In particular, $\mathcal L^2(T)$ is a Hilbert space with the obvious inner product.
When $p=\infty$, the space $\mathcal{L}^\infty(T)$ is the collection of bounded functions $f$ on $T$ equipped with the supremum norm $\|f\|_\infty=\sup\big\{|f(v)|: v\in T\big\}$.  In addition, we say that $q=\infty$ is the conjugate exponent of $p=1$ and that $q=1$ is the conjugate exponent of $p=\infty$.

The interest of the study of operators on infinite trees is motivated mainly by the research in harmonic
analysis dealing with the spectrum of the Laplace operator on discrete structures, one can consult \cite{Ca1} and \cite{Ca2} for more information. Linear operators on discrete structures other than Laplacian have been studied by many authors.  For instance:
\begin{itemize}
\item \emph{Toeplitz operators.}  Pavone studied the properties of  Toeplitz operators on a discrete group and discussed the extent to which they paraller the properties of Topelitz operators on the Hardy space of the unit circle, see \cite{Pa, Pa2}; In \cite{CM}, Colonna and  Mart\'{\i}nez-Avenda\~{n}o defined the Toeplitz operator on $\mathcal L^p$-spaces of a tree, where $1\leqslant p\leqslant \infty$, and then characterized the boundedness and the eigenvalues of such Toeplitz operator; Motivated by the paper \cite{CM}, Zhang and Zhao \cite{ZZ} established several sufficient conditions for Toeplitz operators to be compact on  $\mathcal L^p$-spaces of a tree.
\item \emph{Composition operators.}  Pavone \cite{Pa1} characterized the class of hypercyclic composition operators on the $\mathcal L^p$-space of the Poisson boundary of a homogeneous tree,  and showed that such operators are actually chaotic. Allen, Colonna and Easley \cite{AC} investigated the boundedness, compactness and spectra  of the composition operators on the Lipschitz space of a tree; Later, Allen and Pons (\cite{AP, AP1}) obtained the characterizations for the (weighted) composition operators with closed range, and to be bounded, compact, bounded below, invertible and Fredholm on weighted  Banach spaces of a tree; In addition, the bounded, compact, invertible, isometric composition operators on Hardy spaces of the homogenous rooted trees were investigated in the recent paper \cite{Mu}.
\item \emph{Shift operators.} The operator-norm estimate, spectral properties, hyponormality, subnormality  and hypercyclicity of (weighted) shift operators on certain function spaces of a tree were well-studied  in  \cite{J, Ma1, Ma}; Moreover, the norm estimate, kernels and eigenvalues  of  the shift operator (usually called the adjacency operator) on the $\mathcal L^p$-space of a graph were systematically investigated  in the paper \cite{AB}.
\item {\emph{Multiplication operators.}} The  boundedness, compactness, invertibility, hypercyclicity  and essential norm of multiplication operators on Lipschitz-type spaces of a tree were discussed in a series of papers of Allen, Colonna, Easley and Prudhom \cite{Al2, Al3, Al1, Al4, Co1, Co2}.
\end{itemize}

However, little is known about the spectral properties, self-adjointness and positivity of  classical operators on discrete structures.  The purpose of this paper is to study  some  fundamental properties about Toeplitz operators on  $\mathcal L^p$-spaces of a tree, such as the spectrum, self-adjointness, positivity and finite rank.  Before stating  the definition of the Toeplitz operator on  $\mathcal L^p$-spaces of a tree, we need to introduce the following two important linear operators on such function spaces. For any function $f: T\rightarrow \mathbb C$, the operator $\nabla$ is defined by
\begin{align*}
(\nabla f)(u):=f(u)-\sum_{v^{-}=u}f(v), \ \  \ \ u\in T.
\end{align*}
Recall that the  operator $\Delta$ is  the transformation with domain $\mathcal L^1(T)$ defined by
$$(\Delta f)(u):=\sum_{v\in S_u} f(v), \ \ \ \  f\in \mathcal L^1(T), \ \ \ u\in T,$$
where $S_u$ is the sector determined by $u$. For more details concerning the above two operators, one can consult  \cite{CM}.
Let $1\leqslant p \leqslant \infty$ and $\varphi$ be a complex-valued function on $T$. The \emph{Toeplitz operator} on $\mathcal L^p(T)$ with symbol $\varphi$ is defined by
$$T_\varphi f:=\Delta(\varphi \nabla f), \ \ \ \ f\in \mathcal L^p(T).$$
Some useful properties and important conclusions about  $\nabla$, $\Delta$ and  Toeplitz operators on $\mathcal L^p(T)$ will be included in the next section.

The main part of this paper is organized as follows.  In Section \ref{spectra}, we will study the spectra of some classes of bounded Topelitz operators on the Banach spaces $\mathcal L^p(T)$ with $1\leqslant p\leqslant \infty$. More precisely, we will show in Theorems \ref{spectrum} and \ref{spectrum2} that the spectrum of the Topelitz operator is equal to the closure of the range of its symbol under certain  mild assumptions. This answers  an open question posed by Colonna and  Mart\'{\i}nez-Avenda\~{n}o \cite{CM}. In Section \ref{positivity}, we investigate the self-adjointness and positivity of Topelitz operators on  the Hilbert space $\mathcal L^2(T)$. Surprisingly, we find that there is no nontrivial  self-adjoint Topelitz operator on  $\mathcal L^2(T)$, see Theorem \ref{SP} for the detailed discussion. Section \ref{finite rank}  is devoting to the characterization for  finite rank Toeplitz operators on $\mathcal L^p(T)$ with $1\leqslant p\leqslant \infty$. In fact, we will  prove that the Toeplitz operator on $\mathcal L^p(T)$ is of finite rank if and only if the support of its symbol is a finite set, see Theorem \ref{FR} for the proof.

\section{Preliminary}\label{S2}
In this section, we recall some important definitions and known results which will be required  in the next three sections. We shall make use of the  following derivative of a function on a tree repeatedly.
\begin{defi}\label{f'}
Given a complex-valued function $f$ on a tree $T$.  We define the  function $f_{-}$  on  $T$ by
\begin{align*}
f_{-}(v)=
\begin{cases}
f(v^{-}), & \mathrm{if} \ \ \ v\neq o, \vspace{2mm}\\
0, & \mathrm{if} \ \ \ v=o.
\end{cases}
\end{align*}
In addition, the function $f'$ on $T$ is defined by
$$f'(v)=f(v)-f_{-}(v), \ \ \  v\in T.$$
\end{defi}

Recall that the operators $\nabla$ and $\Delta$ were introduced in Section \ref{Intro}. The relationship between these two operators on $\mathcal L^1(T)$ is given by the following lemma, see \cite[Proposition 5.7]{CM} if necessary.
\begin{lem}\label{DN}
If $f\in \mathcal L^{1}(T)$, then $\Delta(\nabla f)=f$ and $\nabla(\Delta f)=f.$
\end{lem}

The following expression for the Toeplitz operator on $\mathcal L^p(T)$ is quite useful in the study of the boundedness, compactness, positivity and spectra of Toeplitz operators,  which was proved in \cite[Lemma 6.2]{CM}.
\begin{lem}\label{T}
Let $1<p\leqslant \infty$ and let $q$ be the conjugate exponent of $p$.
 If $\varphi$ and $\varphi'$ are both in $\mathcal L^q(T)$ and $f\in \mathcal L^p(T)$, then we have that $\varphi \nabla f\in \mathcal L^1(T)$ and
$$T_\varphi f=f\varphi_{-}+\Delta(f\varphi').$$
\end{lem}

The next two lemmas were established in  \cite[Theorems 6.4 and 6.6]{CM}, which give nice sufficient conditions for the boundedness of Toeplitz operators on $\mathcal L^p(T)$ with $1\leqslant p \leqslant \infty$.

\begin{lem}\label{TB}
Assume that $1\leqslant p<\infty$ and $q$ is the conjugate exponent of $p$. If the functions $\varphi$ and $\widehat{\varphi}$ are both in $\mathcal L^q(T)$, then the Toeplitz operator $T_\varphi$ is bounded on $\mathcal L^p(T)$, where $\widehat{\varphi}$ is defined by $$\widehat{\varphi}(v):=(|v|+1)\varphi'(v), \ \ \ \ v\in T.$$
\end{lem}

\begin{lem}\label{TB2}
If $\varphi$ and $\varphi'$ are both in $\mathcal L^1(T)$,  then the Toeplitz operator $T_\varphi$ is bounded on $\mathcal L^\infty(T)$.
\end{lem}

Recall that \cite[Theorem 6.10]{CM} obtained a characterization on the point spectrum (the set of eigenvalues) of the Toeplitz operator on $\mathcal L^p(T)$ in terms of  the range of its symbol.  For the sake of convenience, we quote this theorem in the following.
\begin{lem}\label{PS}
Let $1\leqslant p\leqslant \infty$ and let $q$ be the conjugate exponent of $p$.
Suppose that $\varphi, \varphi'\in \mathcal L^q(T)$ and the Toeplitz operator $T_\varphi$ is bounded on $\mathcal L^p(T)$. Then:\\
$(a)$ $\sigma_p(T_\varphi)=\{\varphi(w): w\in T\}$ if there is no nonzero function $g\in \mathcal L^p(T)$ with $\nabla g=0$;\\
$(b)$ $\sigma_p(T_\varphi)=\{0\}\cup\{\varphi(w): w\in T\}$ if there is a nonzero function $g\in \mathcal L^p(T)$ with $\nabla g=0$.
\end{lem}

\section{The spectra of Toeplitz operators}\label{spectra}

The study of the spectral properties of operators on infinite trees is relatively new. In this section, we investigate the spectra of Toeplitz operators on  $\mathcal L^p(T)$ with $1\leqslant p\leqslant \infty$. Let $q$ be the conjugate exponent of $p$ and $\varphi$ be a function in $\mathcal L^q(T)$. Based on the characterization for  eigenvalues of Toeplitz operators (Lemma \ref{PS}), we show that the spectrum of the Toeplitz operator $T_\varphi: \mathcal L^p(T)\rightarrow \mathcal L^p(T)$ with $1\leqslant p<\infty$ equals the closure of $\varphi(T)$ if $\widehat{\varphi}$ is also in $\mathcal L^q(T)$. Recall that the function  $\widehat{\varphi}$ was  introduced in Lemma \ref{TB}. On the other hand, we also prove that the spectrum of the Toeplitz operator $T_\varphi: \mathcal L^\infty(T)\rightarrow \mathcal L^\infty(T)$ is coincided with the closure of $\varphi(T)$ when $\varphi'\in \mathcal L^1(T)$.

Let us consider the case that $1\leqslant p<\infty$ first. As the spectrum of the Topelitz operator on $\mathcal L^1(T)$ was obtained in \cite{CM} via the multiplication operator with the same symbol, it is sufficient for us to discuss the case of $1<p<\infty$.

\begin{thm}\label{spectrum}
Let $1\leqslant p<\infty$ and  $q$ be the conjugate exponent of $p$. Suppose that $\varphi$ and $\widehat{\varphi}$ are both in  $\mathcal L^{q}(T)$,  where $$\widehat{\varphi}(v)=(|v|+1)\varphi'(v), \ \ \ \ v\in T.$$ Then $T_\varphi$ is bounded on $\mathcal L^p(T)$ and in this case we have $$\sigma(T_{\varphi})=\mathrm{clos}\{\varphi(v):v\in T\}.$$
\end{thm}

\begin{proof}
As we mentioned above, the conclusion for the case of $p=1$ was obtained in Section 7 of \cite{CM}. So we need only to consider the case for $p>1$. By Lemma \ref{TB}, we see that the Toeplitz operator $T_\varphi$ is bounded on $\mathcal L^p(T)$. Clearly, $\widehat{\varphi}\in \mathcal L^q(T)$ implies that $\varphi'\in \mathcal L^q(T)$. Then using Lemma \ref{PS}, we have
$$\mathrm{clos}\{\varphi(w):w\in T\}=\mathrm{clos}[\sigma_p(T_\varphi)]\subset \sigma(T_\varphi).$$
To prove the reverse inclusion, it is sufficient to show that $T_\varphi-\lambda I$ is invertible on $\mathcal L^p(T)$ if $\lambda$ is not in the closure of the range of $\varphi$,  where $I$ is the identity operator on $\mathcal L^p(T)$.  Observe that $\lambda \notin \mathrm{clos}\{\varphi(v):v\in T\}$ implies that $\lambda$ is not an eigenvalue of $T_\varphi$. This means that $T_\varphi-\lambda I$ is an injective from $\mathcal L^p(T)$ to $\mathcal L^p(T)$.  According to the Banach inverse operator theorem, we need only to show that $T_\varphi-\lambda I$ is a surjective on $\mathcal L^p(T)$  if $\lambda\notin \mathrm{clos}\{\varphi(v):v\in T\}$.

Suppose that $$\inf_{v\in T}|\varphi(v)-\lambda|\geqslant \delta $$
for some positive constant $\delta$. Let $\delta_0=\min\{\delta, |\lambda|\}$. Then we have that  $\delta_0>0$ and
$$|\varphi(v)-\lambda|\geqslant \delta_0$$
and $$ |\varphi_-(v)-\lambda|\geqslant \delta_0$$
for all $v\in T$.  Letting  $g$ be a function in $\mathcal L^p(T)$, we define
\begin{align}\label{f}
f:=\frac{1}{\lambda}\bigg[\Delta\Big(\frac{\varphi}{\varphi-\lambda}\nabla g\Big)-g\bigg].
\end{align}
In the following, we will  show that $f\in \mathcal L^p(T)$ and $f$ is the preimage of $g$ under $T_\varphi-\lambda I$.

To show $f\in \mathcal L^p(T)$,  we need only to prove that
$$\Delta\Big(\frac{\varphi}{\varphi-\lambda}\nabla g\Big)\in \mathcal L^p(T).$$
Let us show $\widehat{\big(\frac{\varphi}{\varphi-\lambda}\big)}\in \mathcal L^{q}(T)$ first.  Noting that
$$\widehat{\Big(\frac{\varphi}{\varphi-\lambda}\Big)}(v)=(|v|+1)\Big(\frac{\varphi}{\varphi-\lambda}\Big)'(v)=
(|v|+1)\Big[\frac{\varphi(v)}{\varphi(v)-\lambda}-\frac{\varphi_-(v)}{\varphi_-(v)-\lambda}\Big]$$
for $v\in T$,
we have
\begin{align*}
\bigg\|\widehat{\Big(\frac{\varphi}{\varphi-\lambda}\Big)}\bigg\|_{q}&=\bigg[\sum\limits_{v\in T}\Big|(|v|+1)\Big(\frac{\varphi(v)}{\varphi(v)-\lambda}-\frac{\varphi_-(v)}{\varphi_-(v)-\lambda}\Big)\Big|^{q}\bigg]^{\frac{1}{q}}\\
&=\bigg[\sum\limits_{v\in T}\Big|(|v|+1)\Big(\frac{\varphi(v)}{\varphi(v)-\lambda}-\frac{\varphi_-(v)}{\varphi(v)-\lambda}+
\frac{\varphi_-(v)}{\varphi(v)-\lambda}-\frac{\varphi_-(v)}{\varphi_-(v)-\lambda}\Big)\Big|^{q}\bigg]^{\frac{1}{q}}.
\end{align*}
Then triangle inequality gives that
\begin{align*}
\bigg\|\widehat{\Big(\frac{\varphi}{\varphi-\lambda}\Big)}\bigg\|_{q}&\leqslant \bigg[\sum\limits_{v\in T}\Big|(|v|+1)\frac{\varphi(v)-\varphi_-(v)}{\varphi(v)-\lambda}\Big|^{q}\bigg]^{\frac{1}{q}}+
\bigg[\sum\limits_{v\in T}\Big|(|v|+1)\frac{\varphi_-(v)[\varphi(v)-\varphi_-(v)]}{[\varphi_-(v)-\lambda][\varphi(v)-\lambda]}\Big|^{q}\bigg]^{\frac{1}{q}}\\
&=\bigg[\sum\limits_{v\in T}\Big|(|v|+1)\frac{\varphi'(v)}{\varphi(v)-\lambda}\Big|^{q}\bigg]^{\frac{1}{q}}+
\bigg[\sum\limits_{v\in T}\Big|(|v|+1)\frac{\varphi'(v)\varphi_-(v)}{[\varphi_-(v)-\lambda][\varphi(v)-\lambda]}
\Big|^{q}\bigg]^{\frac{1}{q}}\\
&\leqslant \frac{1}{\delta_0}\Big(\sum\limits_{v\in T}\big|(|v|+1)\varphi'(v)\big|^{q}\Big)^{\frac{1}{q}}+
\frac{1}{\delta_0^{2}}\Big(\sum\limits_{v\in T}\big|(|v|+1)\varphi_-(v)\varphi'(v)\big|^{q}\Big)^{\frac{1}{q}}\\
&\leqslant \frac{1}{\delta_0}\Big(\sum\limits_{v\in T}\big|(|v|+1)\varphi'(v)\big|^{q}\Big)^{\frac{1}{q}}+
\frac{\|\varphi\|_\infty}{\delta_0^{2}}\Big(\sum\limits_{v\in T}\big|(|v|+1)\varphi'(v)\big|^{q}\Big)^{\frac{1}{q}}\\
&=\Big(\frac{1}{\delta_0}+\frac{\|\varphi\|_\infty}{\delta_0^{2}}\Big)\|\widehat{\varphi}\|_{q}.
\end{align*}
 This shows that $\widehat{\big(\frac{\varphi}{\varphi-\lambda}\big)}\in \mathcal L^{q}(T)$, since  $\widehat{\varphi}\in\mathcal  L^{q}(T)$.

Next, we are going to show that $\Delta\big(\frac{\varphi}{\varphi-\lambda}\nabla g\big)\in \mathcal L^p(T)$.  Noting that $\frac{\varphi}{\varphi-\lambda}$ and $\big(\frac{\varphi}{\varphi-\lambda}\big)'$ are both  in $\mathcal L^q(T)$, we obtain by Lemma \ref{T} that
$$\Delta\Big(\frac{\varphi}{\varphi-\lambda}\nabla g\Big)=T_{\frac{\varphi}{\varphi-\lambda}}g=\Big(\frac{\varphi}{\varphi-\lambda}\Big)_{-}g+\Delta\Big[\Big(\frac{\varphi}{\varphi-\lambda}\Big)'g\Big].$$
Furthermore, we have
\begin{align*}
\Big\|\Delta\Big[\Big(\frac{\varphi}{\varphi-\lambda}\Big)'g\Big]\Big\|_p & \leqslant \Big\|\Delta\Big[\Big(\frac{\varphi}{\varphi-\lambda}\Big)'g\Big]\Big\|_1=\sum_{v\in T} \bigg| \Delta\Big[\Big(\frac{\varphi}{\varphi-\lambda}\Big)'g\Big](v)\bigg|\\
&\leqslant \sum_{v\in T}\sum_{u\in S_v}\bigg| \Big(\frac{\varphi}{\varphi-\lambda}\Big)'(u)g(u)\bigg|\\
&=\sum_{v\in T}(|v|+1)\bigg| \Big(\frac{\varphi}{\varphi-\lambda}\Big)'(v)g(v)\bigg|\\
&=\sum_{v\in T}\bigg| \widehat{\Big(\frac{\varphi}{\varphi-\lambda}\Big)}(v)\bigg|~|g(v)|\\
&\leqslant \|g\|_p\bigg\|\widehat{\Big(\frac{\varphi}{\varphi-\lambda}\Big)}\bigg\|_{q},
\end{align*}
to get that $\Delta\big[\big(\frac{\varphi}{\varphi-\lambda}\big)'g\big]\in \mathcal L^p(T)$. Since $\frac{\varphi}{\varphi-\lambda}\in \mathcal L^{q}(T)$ implies that  $\frac{\varphi}{\varphi-\lambda}\in \mathcal L^\infty(T)$, we obtain that $\big(\frac{\varphi}{\varphi-\lambda}\big)_{-}$ is also bounded. It follows that $\big(\frac{\varphi}{\varphi-\lambda}\big)_{-}g$ is in $\mathcal L^p(T)$, which gives  that $\Delta\big(\frac{\varphi}{\varphi-\lambda}\nabla g\big)\in \mathcal L^p(T)$. Therefore, we obtain that the function defined in (\ref{f}) is in $\mathcal L^p(T)$.

To finish the proof, it remains to verify that
$$(T_{\varphi}-\lambda I)f=g.$$
Indeed, we have by Lemma \ref{T} that $$\frac{\varphi}{\varphi-\lambda}\nabla g\in \mathcal L^{1}(T),$$since $g\in \mathcal L^p(T)$ and $\frac{\varphi}{\varphi-\lambda}, \big(\frac{\varphi}{\varphi-\lambda}\big)' \in \mathcal L^q(T)$. It follows from (\ref{f}) and Lemma \ref{DN} that
\begin{align*}
\nabla f&=\nabla\Bigg[\frac{1}{\lambda}\bigg(\Delta\Big(\frac{\varphi}{\varphi-\lambda}\nabla g\Big)-g\bigg)\Bigg]=\frac{1}{\lambda}\frac{\varphi}{\varphi-\lambda}\nabla g-\frac{\nabla g}{\lambda}=\frac{\nabla g}{\varphi-\lambda}.
\end{align*}
This yields that
\begin{align*}
(T_{\varphi}-\lambda I)f&=T_{\varphi}f-\lambda f=\Delta(\varphi\nabla f)-\lambda f\\
&=\Delta\Big(\frac{\varphi}{\varphi-\lambda}\nabla g\Big)-\Delta\Big(\frac{\varphi}{\varphi-\lambda}\nabla g\Big)+g\\
&=g,
\end{align*}
as desired. This completes the proof of Theorem \ref{spectrum}.
\end{proof}

Now we give a description for the spectrum of $T_\varphi: \mathcal L^\infty(T)\rightarrow \mathcal L^\infty(T)$ with $\varphi, \varphi'\in \mathcal L^1(T)$ in the following theorem.

\begin{thm}\label{spectrum2}
If $\varphi$ and $\varphi'$ are both in $\mathcal L^1(T)$,  then $T_\varphi$ is bounded on $\mathcal L^\infty(T)$ and in this case we have $$\sigma(T_{\varphi})=\mathrm{clos}\{\varphi(v):v\in T\}.$$
\end{thm}

\begin{proof}
Lemma \ref{TB2} tells us that $T_\varphi: \mathcal L^\infty(T)\rightarrow \mathcal L^\infty(T)$ is bounded if $\varphi\in \mathcal L^1(T)$ and $\varphi'\in \mathcal L^1(T)$.
According to the proof of Theorem \ref{spectrum}, we need only to show that
$$\Big(\frac{\varphi}{\varphi-\lambda}\Big)_{-}g\ \ \ \  \mathrm{and} \ \ \ \  \Delta\Big[\Big(\frac{\varphi}{\varphi-\lambda}\Big)'g\Big]$$
are both in $\mathcal L^\infty(T)$
if $\lambda\notin \mathrm{clos}\{\varphi(v): v\in T\}$ and $g\in \mathcal L^1(T)$.  In fact,
$$\bigg\|\Big(\frac{\varphi}{\varphi-\lambda}\Big)_{-}g\bigg\|_\infty=\sup_{v\in T}\bigg|\Big(\frac{\varphi}{\varphi-\lambda}\Big)_{-}(v)g(v)\bigg|\leqslant \frac{\|g\|_\infty}{\delta_0}\|\varphi\|_\infty,$$
where $\delta_0$ is the constant defined in the proof of Theorem \ref{spectrum}. Moreover, since
\begin{align*}
\bigg\|\Delta\Big[\Big(\frac{\varphi}{\varphi-\lambda}\Big)'g\Big]\bigg\|_\infty&=\sup_{v\in T}\bigg|\sum_{u\in S_v}g(u)\Big(\frac{\varphi}{\varphi-\lambda}\Big)'(u)\bigg|\\
&\leqslant \sup_{v\in T}\sum_{u\in S_v}|g(u)|\bigg|\frac{\varphi(u)}{\varphi(u)-\lambda}-\frac{\varphi_-(u)}{\varphi_-(u)-\lambda}\bigg|\\
&\leqslant \frac{|\lambda|~\|g\|_\infty}{\delta_0^2}\sup_{v\in T}\Big[\sum_{u\in S_v}|\varphi(u)-\varphi_{-}(u)|\Big]\\
&=  \frac{|\lambda|~\|g\|_\infty}{\delta_0^2}\|\varphi'\|_1,
\end{align*}
we obtain that $\big(\frac{\varphi}{\varphi-\lambda}\big)_{-}g\in \mathcal L^\infty(T)$ and   $\Delta\big[\big(\frac{\varphi}{\varphi-\lambda}\big)'g\big]\in \mathcal L^\infty(T)$.

The rest of the proof of this theorem is similar to the previous one, so we omit the details. This finishes the proof of Theorem \ref{spectrum2}.
\end{proof}

\begin{rem}
The conclusions in Theorems \ref{spectrum} and \ref{spectrum2} are comparable to the corresponding results for analytic Toeplitz operators on the Hardy and Bergman spaces of the open unit disk, see the books \cite{Dou} and \cite{Zhu} for more information.
\end{rem}

\section{The positivity of Toeplitz operators}\label{positivity}
 We say that a bounded linear operator $T$ is self-adjoint (positive) on a complex  Hilbert space $\mathscr H$ if
 $\langle Tx, x\rangle_{\mathscr H}$
 is real (nonnegative) for every $x$ in $\mathscr H$. In this section, we study when a bounded Toeplitz operator  is self-adjoint (positive) on the Hilbert space $\mathcal L^2(T)$.

 Observing that the adjoint of the Topelitz operator $T_\varphi$ does not equal $T_{\overline{\varphi}}$ in general, which leads to certain difficulties in the study of the self-adjointness and positivity of Toeplitz operators on $\mathcal L^2(T)$. However, we are able to obtain a complete characterization for the self-adjoint Toeplitz operators on $\mathcal L^2(T)$ by  using their point spectra.

\begin{thm}\label{SP}
Suppose that $\varphi$ is a function in $\mathcal L^2(T)$ such that $\widehat{\varphi}\in \mathcal L^2(T)$. Then the following three conditions  are equivalent:\\
$(1)$ $T_\varphi$ is positive on $\mathcal L^2(T)$;\\
$(2)$ $T_\varphi$ is self-adjoint on $\mathcal L^2(T)$;\\
$(3)$ $\varphi(v)=0$ for all $v\in T$.
\end{thm}

\begin{proof}
Clearly, $(1)\Rightarrow (2)$ is trivial and $(3)\Rightarrow (1)$ follows immediately from Lemma \ref{T}. To complete the proof of this theorem, it remains to show that $(2)\Rightarrow (3)$.

Now we suppose that $T_\varphi$ is self-adjoint on $\mathcal L^2(T)$. Then we have by Lemma \ref{PS} that $\varphi$ must be real-valued, since
$$\{\varphi(v): v\in T\}=\sigma_p(T_\varphi)\subset \sigma(T_\varphi)\subset \mathbb R$$
or
$$\{0\}\cup \{\varphi(v): v\in T\}=\sigma_p(T_\varphi)\subset \sigma(T_\varphi)\subset \mathbb R.$$

Let $f$ be any function in $\mathcal L^2(T)$. Since $\widehat{\varphi}\in \mathcal L^2(T)$,  we see  that the series $\sum\limits_{v\in T}[T_{\varphi}f(v)]\overline{f(v)}$ is absolutely convergent. Now  using Lemma \ref{T} again gives
\begin{align}\label{T_phi}
\begin{split}
\langle T_{\varphi}f, f\rangle&=\sum\limits_{v\in T}[T_{\varphi}f(v)]\overline{f(v)}=\sum\limits_{v\in T}\big[f(v)\varphi_{-}(v)\overline{f(v)}+\Delta(f\varphi')(v)\overline{f(v)}\big]\\
&=\sum\limits_{v\in T}|f(v)|^{2}\varphi_-(v)+\sum\limits_{v\in T}\overline{f(v)}\Delta(f\varphi')(v).
\end{split}
\end{align}
Furthermore, we have by the definitions of  $\varphi'$ and the operator $\Delta$ that
\begin{align}\label{sum}
\begin{split}
&\sum\limits_{v\in T}\overline{f(v)}\Delta(f\varphi')(v)\\
&\ =\sum\limits_{v\in T}\Big[\overline{f(v)}\sum\limits_{u\in S_{v}}(f\varphi')(u)\Big]\\
&\ =\sum\limits_{v\in T}\Big\{\overline{f(v)}\sum\limits_{u\in S_{v}}f(u)\big[\varphi(u)-\varphi_-(u)\big]\Big\}\\
&\ =\sum\limits_{v\in T}|f(v)|^{2}\varphi(v)-\sum\limits_{v\in T}|f(v)|^{2}\varphi_-(v)+\sum\limits_{v\in T}\Big\{\overline{f(v)}\sum\limits_{u\in S_{v}\backslash\{v\}}\big[f(u)\varphi(u)-f(u)\varphi_-(u)\big]\Big\}.
\end{split}
\end{align}

We first show that $\varphi$ is a constant function. Otherwise, we can choose two vertices $\xi$ and $\eta$ in $T$ such that $\varphi(\xi)\neq \varphi(\eta)$. Since the path joining $\xi$ and $\eta$ is a finite sequence of  distinct vertices, we may assume that $\xi\thicksim \eta$. In other words, we obtain that $\xi\in \mathrm{Chi}(\eta)$ or $\eta\in \mathrm{Chi}(\xi)$. Without loss of generality, we may assume that $\eta\in \mathrm{Chi}(\xi)$. Let us  consider the function $g: T\rightarrow \mathbb C$ defined by
\begin{align*}
g(v)=
\begin{cases}
1,& \mathrm{if} \ \ v=\xi,\vspace{2mm}\\
\mathrm{i},& \mathrm{if} \ \ v=\eta,\vspace{2mm}\\
0,& \text{otherwise}.\vspace{2mm}\\
\end{cases}
\end{align*}
It follows from (\ref{T_phi}) and (\ref{sum}) that
\begin{align*}
\langle T_{\varphi}g, g\rangle
&=|g(\xi)|^{2}\varphi(\xi)+|g(\eta)|^{2}\varphi(\eta)+\overline{g(\xi)}\big[g(\eta)\varphi(\eta)-g(\eta)\varphi(\eta^{-})\big]\\
&=|g(\xi)|^{2}\varphi(\xi)+|g(\eta)|^{2}\varphi(\eta)+\overline{g(\xi)}g(\eta)\big[\varphi(\eta)-\varphi(\eta^{-})\big]\\
&=\big[\varphi(\xi)+\varphi(\eta)\big]+\big[\varphi(\eta)-\varphi(\xi)\big]\mathrm{i}.
\end{align*}
This means that $\langle T_{\varphi}g, g\rangle\notin \mathbb R$, which contradicts that $T_\varphi$ is self-adjoint. Thus $\varphi$ is a constant function. However, the condition $\varphi\in \mathcal L^2(T)$ yields that $\varphi$ must be zero. This completes the proof.
\end{proof}

\section{Finite rank Toeplitz operators}\label{finite rank}

The final section is devoted to establishing a equivalent characterization for  finite rank Toeplitz operators on $\mathcal L^p(T)$, where $1\leqslant p \leqslant \infty$.  More concretely, we will show  in the following theorem that the Toeplitz  operator $T_\varphi$ having finite rank if and only if $\{v\in T: \varphi(v)\neq 0\}$ is a  finite subset of $T$.

\begin{thm}\label{FR}
Let $1\leqslant p \leqslant \infty$ and $\varphi$ be a function  in $\mathcal L^q(T)$ such that  $\varphi'\in \mathcal L^q(T)$ and the Toeplitz operator $T_\varphi$ is bounded on $\mathcal L^p(T)$, where $\frac{1}{p}+\frac{1}{q}=1$. Then $T_\varphi$ has finite rank on $\mathcal L^p(T)$ if and only if $\varphi(v)=0$ except for finitely  many points $v\in T$.
\end{thm}

\begin{proof}
We show the sufficiency first. Suppose that $\varphi(v)=0$ for all $v\in T$ except $v_1, v_2, \cdots, v_n$.  For any $f\in \mathcal L^p(T)$, we have
\begin{align*}
T_\varphi f(v)=\Delta (\varphi \nabla f)(v)=\sum_{u\in S_v}\varphi(u)\nabla f(u), \ \ \ \  v\in T.
\end{align*}
It follows that $T_\varphi f(v)=0$ for all $v\in T$ with $|v|>\max\{|v_k|: k=1, 2, \cdots, n\}$. Let
$$W:=\Big\{w\in T: |w|\leqslant \max_{1\leqslant k \leqslant n} |v_k|\Big\}.$$
Then $W$ is a finite set and its elements  can be enumerated by
$$W=\{w_1, w_2, \cdots, w_N\}.$$

Fix a vertex $\xi\in T$ and define
 \begin{align}\label{basis}
e_\xi(v)=\begin{cases}
1,& \mathrm{if} \ \ v=\xi,\vspace{2mm}\\
0,& \text{otherwise}.
\end{cases}
\end{align}
Since $T_\varphi f$ is a vector in $\mathcal L^p(T)$, it can be expressed as
$$T_\varphi f=c_1e_{w_1}+c_2e_{w_2}+\cdots+c_Ne_{w_N}$$
with constants $c_1, c_2, \cdots, c_N$.
This implies that
$$\mathrm{range}(T_\varphi)\subset \mathrm{span}\big\{e_{w_1}, e_{w_2},\cdots, e_{w_N}\big\},$$
to obtain that $T_\varphi$ is of finite rank on $\mathcal L^p(T)$.

To show the necessity, we suppose that $T_\varphi$ is a finite rank operator on $\mathcal L^p(T)$ and $$\mathrm{range}(T_\varphi)\subset \mathrm{span}\big\{e_{v_1}, e_{v_2},\cdots, e_{v_n}\big\},$$
where  $v_k$ is a vertex of $T$ and $e_{v_k}$ is defined in (\ref{basis}). According to Lemma \ref{T} and the computations in (\ref{sum}), we obtain that
\begin{align}\label{ew}
\begin{split}
T_{\varphi}e_{w}(v)&=e_{w}(v)\varphi(v^-)+\sum_{u\in S_{v}}e_{w}(u)[\varphi(u)-\varphi(u^-)]\\
&=e_{w}(v)\varphi(v^-)+e_{w}(v)[\varphi(v)-\varphi(v^-)]+\sum_{u\in S_{v}\backslash\{v\}}e_{w}(u)[\varphi(u)-\varphi(u^-)]\\
&=e_{w}(v)\varphi(v)+\sum_{u\in S_{v}\backslash \{v\}}e_{w}(u)[\varphi(u)-\varphi(u^-)]
\end{split}
\end{align}
for every $v\neq o$ and $w\in T$.

If $\varphi$ does  not vanish at infinitely many vertices  of $T$, then we can choose a vertex $w\in T$ such that
$|w|>\max\big\{|v_{1}|, |v_2|, \cdots, |v_n|\big\}$ and  $\varphi(w)\neq 0$. It follows from (\ref{ew}) that
$$T_{\varphi}e_w(w)=\varphi(w).$$
We claim that $T_{\varphi} e_w\notin \mathrm{span}\big\{e_{v_1}, e_{v_2},\cdots, e_{v_n}\big\}$. Indeed, if
$$T_{\varphi} e_w=c_1e_{v_1}+c_2e_{v_2}+\cdots+c_ne_{v_n}$$
for some scalars $c_1, c_2, \cdots, c_n$, then we would have
$$0=c_1e_{v_1}(w)+c_2e_{v_2}(w)+\cdots+c_ne_{v_n}(w)=T_{\varphi} e_w(w)=\varphi(w)\neq 0,$$
which is a contradiction.  This yields that $$\mathrm{dim}\big[\mathrm{range}(T_\varphi)\big]>n.$$
But this contradicts that the dimension of $\mathrm{range}(T_\varphi)$ is less than or equal to $n$. Therefore,  the proof of Theorem \ref{FR} is finished.
\end{proof}

\begin{rem}
It is worth noting that Theorem \ref{FR} is analogous to the corresponding results for finite rank Toeplitz operators on the Bergman and Fock spaces, see \cite{Luc} and \cite[Theorem 6.42]{Zhu2} respectively.
\end{rem}
\vspace{2mm}

\subsection*{Acknowledgment}
This work was partially supported by  NSFC (grant number: 11701052) and Chongqing Natural Science Foundation (cstc2019jcyj-msxmX0337). The third author was partially supported by the Fundamental Research Funds for the Central Universities (grant numbers: 2020CDJQY-A039, 2020CDJ-LHSS-003).\vspace{5mm}\\

\end{document}